\documentclass[a4paper,12pt,twoside,reqno]{amsart} 
\DeclareSymbolFontAlphabet{\mathcal}{symbols}
\usepackage[colorlinks=red,linkcolor=blue]{hyperref}
\usepackage[cmtip,all]{xy}
\usepackage{amsmath}
\usepackage[psamsfonts]{amssymb}
\usepackage[psamsfonts]{eucal}
\usepackage{amsthm}
\usepackage[psamsfonts]{amssymb}
\usepackage[scaled]{helvet}
\usepackage{mathrsfs}
\usepackage{bm}
\usepackage[a4paper]{geometry}
\usepackage{multicol}
\usepackage{enumerate}
\usepackage{graphicx}
\usepackage{tikz}
\usepackage{color}
\definecolor{trama}{gray}{.875}
\usepackage{lmodern}
\usepackage{textcomp}
\usepackage{enumerate}
\usepackage{wrapfig}
\title{Lie models for nilpotent spaces}
\date{\today}

\author[Y. F\'elix]{Yves F\'elix}
\address{Institut de Math\'ematiques et Physique\\
         Universit\'e Catholique de Louvain-la-Neuve\\
         Louvain-la-Neuve\\
         Belgique}
\email{Yves.felix@uclouvain.be}

\author[J. M. Moreno-Fern\'{a}ndez]{Jos\'{e} M. Moreno-Fern\'{a}ndez}
\address{Departamento de Algebra, Geometr\'{\i}a y Topolog\'{\i}a, Universidad
de M\'alaga, Ap. 59, 29080 M\'alaga, Spain}
\email{josemoreno@uma.es}

\author[D. Tanr\'e]{Daniel Tanr\'e}
\address{D\'epartement de Math{\'e}matiques\\
         UMR 8524 \\
         Universit\'e de Lille~1\\
         59655 Villeneuve d'Ascq Cedex\\
         France}
\email{Daniel.Tanre@univ-lille1.fr}

\thanks{The authors have been partially supported by the MINECO grants MTM2013-41768-P and MTM2016-78647-P. The second author has also been partially supported by the Junta de Andaluc{\'\i}a grant FQM-213.
}



\theoremstyle{plain}

\newtheorem{proposition}{Proposition} 
\newtheorem{theoremb}{Theorem}
\newtheorem{lemma}[proposition]{Lemma}

\theoremstyle{definition}
\newtheorem{definition}[proposition]{Definition}

\theoremstyle{remark}
\newtheorem{remark}[proposition]{Remark}

\newcommand{\thmref}[1]{Theorem~\ref{#1}}

\newcommand{\remref}[1]{Remark~\ref{#1}}

\def\cC{{\mathcal C}}

\def\cL{{\mathcal L}}

\def\L{\mathbb{L}}

\def\Q{\mathbb{Q}}

\def\hL{{\widehat{\mathbb L}}}

\def\ad{{\rm ad}}

\def\ov{\overline}

    \newcommand{\lasu}{{\mathfrak{L}}}
    \newcommand{\fracd}{{\mathfrak{D}}}

\newcommand{\catcdga}{\operatorname{{\bf CDGA}}}

\newcommand{\CDGC}{\operatorname{{\bf CDGC}}}

\newcommand{\catdglco}{\operatorname{{\bf DGLC}}}
\newcommand{\catdgl}{\operatorname{{\bf DGL}}}

  \newcommand{\libre}{\mathbb L}

\def\Lc{\mathbb{L}^{\!c}}

\begin{document}

\begin{abstract}
Let $(L,d)$ be a differential graded Lie algebra, where $L=\libre(V)$ is free as graded Lie algebra and $V=V_{\geq 0}$ is a finite type graded vector space. We prove that the injection of $(L,d)$ into its completion $(\widehat{L},d)$ is a quasi-isomorphism if and only if $H(L,d)$ is a finite type pronilpotent graded Lie algebra. As a consequence, we obtain an equivalence between graded Lie models for nilpotent spaces in rational homotopy theory.  
\end{abstract}

\maketitle


In this paper, all graded vector spaces and other algebraic structures are defined over the field $\mathbb Q$. A graded vector space is said to be of finite type if it is finite dimensional in each degree. We will denote by $\mathbb L (V)$ the free graded Lie algebra on the graded vector space $V$. 

Let $(L,d)$  be a differential graded Lie algebra (dgl hereafter) where $L =\mathbb L (V),$ $V= V_{\geq 0}$ is of finite type, 
and the differential $d$ decreases the degree by $1$. 
The \emph{completion} of $(L,d)$ is the dgl $(\widehat{\mathbb L}(V),d)$ where 
$\widehat{\mathbb L}(V) = \varprojlim_n \mathbb L (V)/ \mathbb L^{>n}(V)$. Our main result  states

\begin{theoremb}\label{thm:main}   With the notation above, 
  the injection $(\mathbb L(V),d) \to (\widehat{\mathbb L}(V),d)$ induces an isomorphism in homology if,
   and only if, $H(L,d)$ is a finite type pronilpotent graded Lie algebra.
   \end{theoremb}
 
Recall that a graded Lie algebra $L$ is \emph{pronilpotent} if $L= \varprojlim_p L/L^p$, where $L^p$ denotes  
the lower central series of   $L$, defined by:  
$$L^1= L \hspace{5mm}\mbox{and }  L^n=[L, L^{n-1}]\hspace{5mm}\mbox{ for $n>1$.}$$

 \thmref{thm:main} has importance in rational homotopy. Indeed, let $X$ be a finite simplicial complex 
that is connected and nilpotent. 
Then, a rational dgl model for $X$ was constructed by J.~Neisendorfer (\cite{Neis}). 
This dgl is of the form $(\mathbb L(V),d)$, with $V\cong s^{-1}\widetilde H_*(X;\mathbb Q)$ and 
its homology is the rational homotopy Lie algebra of $X$. 
More recently, based on the works of Dupont (\cite{du}), and of Lawrence and Sullivan (\cite{LS}), 
a new dgl model has been defined for any finite simplicial complex (\cite{BFMT}). 
The simplest way to construct it consists in using the transfer diagram 
$$
\xymatrix{
{\fracd}\colon
&
 \ar@(ul,dl)@<-5.5ex>[]_{\phi}
 &
  A_{PL}(X) \ar@<0.75ex>[r]^-{p}
  &
   {C^*(X),} \ar@<0.75ex>[l]^-{i} }
$$
where $C^*(X)$ denotes the rational simplicial cochain complex on $X$. This transfer diagram induces a differential 
$d$ on the graded Lie algebra  $\widehat{\mathbb L}(W)$, where $W=s^{-1}C_{*}(X;\Q)$, 
turning   $\lasu_X:= (\widehat{\mathbb L}(W),d)$ into a dgl.  
Let $x_0$ be a base point in $X$. Then, the quotient   
$\ov{\lasu}_{X}=\lasu_{X}/\lasu_{(x_{0)}}$  
is called the \emph{reduced Sullivan Lie model of $X$ (\cite{Getz,BFMT})}. 
More details can be found in Section \ref{Sec3}. As a corollary of our main theorem, we deduce

\begin{theoremb} Let $X$ be a connected, nilpotent and finite simplicial complex. 
Then,  its Neisendorfer model and its reduced Sullivan Lie model are quasi-isomorphic.
\end{theoremb}

\section{Pronilpotent graded Lie algebras}

We characterize the pronilpotency of a finite type graded Lie algebra in Lemma \ref{lemprel}. 
Next, we prove the sufficiency in Theorem \ref{thm:main} as Proposition \ref{Prop1}. 
In this section, we assume that $L=L_{\geq0}$ is a non negatively graded Lie algebra,
not necessarily free as graded Lie algebra.

\medskip
For each $p\geq 1$, we denote by $G_p^n$ the descending series of $L_0$-modules
$$G_p^1= L_p\,, \hspace{5mm}\mbox{and  } G_p^n= [L_0, G_p^{n-1}]\,\hspace{5mm} \mbox{ for $n>1$}.$$

\begin{lemma} \label{lemprel} A finite type graded lie algebra $L$ is pronilpotent  if, and only if,
\begin{enumerate}[(a)]
\item $L_0$ is nilpotent, and
\item for each $p>0$, there is an integer $n(p)$ such that $G_p^{n(p)}= 0$. 
\end{enumerate}
\end{lemma}

\begin{proof} 
Suppose $L$  pronilpotent. Then we have $\cap_n L^n= 0$. Therefore, if $(L^n)_q\neq 0$, then $(L^n)_q$ contains strictly $(L^{n+1})_q$. Since dim$\, L_q<\infty$, this implies (a) and (b). 

Conversely, suppose (a) and (b)  satisfied. Then for each integer $p$, there is some $n$ with $L^n = (L^n)_{\geq p}$. 
This implies $\cap_n L^n= 0$, and the canonical map $\varphi \colon L\to \varprojlim_n L/L^n$ is injective. Now, for each degree $p$, there is some $n(p)$ such that
$$(\varprojlim_n L/L^n)_p = (L/L^{n(p)})_p\,.$$
Therefore, $\varphi$ is surjective and $L$ is pronilpotent. 
\end{proof}

Now, we prove the first part of our main theorem. Let $(L,d)= (\mathbb L(V),d)$ be a dgl where $V= V_{\geq 0}$ is a finite type graded vector space. 

\begin{proposition}\label{Prop1} 
If the injection $(L,d)\to (\widehat{L},d)$ is a quasi-isomorphism, then 
$H(L,d)$ is a pronilpotent Lie algebra.
\end{proposition}

\begin{proof} Write $E = H(L,d)$, and $I^p$ for the image of $H(L^p)$ in $E$. We first prove that $E = \varprojlim_p E/I^p$.
 To do so, we start by proving that $\cap_p I^p= 0$. Assume that $0\neq [a]\in \cap_pI^p$. Since an element of $I^p$ is
  represented by a cycle in $L^p$, there are cycles $a_n\in L^n$ with $[a_n]= [a]$, and elements $v_n\in L$ with 
  $dv_n = a-a_n$. It follows that $dv_n= a$ in $L/L^n$. 

Fix some integer $n$ and for $q>n$, let $V_q\subset L/L^q$ be the set of elements $x$ such that $dx=a$.
Denote by $V_{n,q}$ the image of $V_q$ in $L/L^n$ along the projection $L/L^q\to L/L^n$. Then, the sequence
$$V_{n,n}\supset V_{n, n+1}\supset \dots$$
is a sequence of non empty finite dimensional affine spaces. It must thus stabilize. 
Therefore, there is a sequence of elements $ b_q\in L/L^q $ such that $db_q=a$ in $L/L^q$ and $b_q=b_{q-1}$ 
in $L/L^{q-1}$. The sequence $(b_q)$ determines an element $b$ in $\widehat L,$ and since $db_q=a$ in 
$L/L^q$ for all $q$, $db=a$, i.e., $[a]= 0$. 
This implies that $\cap_n I^n= 0$. Therefore, the natural map $\varphi\colon E\to  \varprojlim_p E/I^p$ is injective. 

The second step consists in proving that $\varphi\colon E\to \varprojlim_p E/I^p$ is surjective. Let $(a_q)$ be an element in $\varprojlim_p E/I^p$. Then the $a_q$ are cycles in $L$, and there exist elements $b_q\in L$   with $db_q =  a_{q+1}-a_q $ in $L/L^q$. We denote by $V_n$ the image of $a_n+ d(L)$ in $L/L^n$. In each degree the space $V_n$ is a finite dimensional affine space. By definition of the sequence $a_n$, for $q>n$, the restriction map $L/L^q \to L/L^{n}$ maps $
V_q$ into $V_{n}$. We denote the image by $V_{n,q}$. Since for each degree, the sequence $$\dots \subset V_{n,q}\subset V_{n, q-1}  \subset \dots \subset V_n$$
stabilizes, there is a sequence of cycles $u_n \in L/L^n$ such that $u_n = u_{n-1}$ in $L/L^{n-1}$ and $[u_n]= [a_n]$ 
in $H(L/L^n)$. It follows directly that $[a_n]=[u_n]$ in $E/I_n$. In $\widehat{L}$, 
the sequence $u_n$ lifts to a cycle $u$, and   in $E/I_n$, $[u]=[a_n]$. 
Since $H(L)\cong H(\widehat{L})$ we have a cycle $v\in L$ with $\varphi (v) = (a_n)$, 
and the surjectivity of $\varphi$ is proven. 

Now, assume that for some $q$, dim$\, E_q=\infty$. Let $W^{n}$ be a supplement of $I^{n+1}_q$ in $I_q^n$, 
$$I_q^n = W^n \oplus I_q^{n+1}\,.$$
Then, $$ {E_q} \cong \prod W^n$$ as graded vector spaces. 
Since $L$ is finite type, dim$\, W^n<\infty$ for all $n$, and $\sum_n \mbox{dim} W^n=\infty$. 
Therefore  ${E}_q$ is an uncountable vector space, but this is impossible since ${E}_q$ injects as a graded vector space 
into  $L_q$, which is a countable vector space. 
It follows that $E_q$ is a finite dimensional vector space for all $q$, 
and that for each $q$ there is an integer $n(q)$ such that $I_q^{n(q)}= 0$.

We have proven the nilpotency of $E_0$, because a class in $E_0^{n(0)}$ is represented by a cycle in $L^{n(0)}$ and,
by definition of $n(0)$, such a cycle is a boundary in $L$.  
The same argument proves that the sequence $G^n_q\supset G_q^{n+1}$  stabilizes at $0$. 
By Lemma \ref{lemprel}, the result follows. 
\end{proof}

\begin{remark}\label{rem:simple}
 \thmref{thm:main} implies that any dgl whose homology is not finite type or pronilpotent is not quasi-isomorphic to its completion. 
 For instance, let $L=\mathbb L (x,y,z)$ with $\vert x\vert= \vert y\vert = 0$, $\vert z\vert = 1$, $dx=dy=0$ and $dz=x-[y,x]$.
  Then, $H(L)_0$ is the 2-dimensional Lie algebra on $x$ and $y$, with bracket $[y,x]=0$.
 Meanwhile in the completion the element $x$ is a boundary, $x= d(\sum_{q\geq 0} \ad^q_y(z))$, and therefore
 $H(\widehat L)_0$ $\cong \mathbb Q  y$.
 \end{remark}

\section{The Quillen   functors $\mathcal C$ and $\mathcal L$ and the proof of the main Theorem}

Let $\catcdga$ be the category of augmented commutative differential graded algebras (cdga's hereafter) $(A,d)$ concentrated in nonnegative degrees with a differential of degree $+1$ and whose cohomology is finite dimensional in each degree, with $H^0(A)= \mathbb Q$.  

The free graded commutative algebra on a graded vector space $V = V^{\geq 1}$, denoted by  $\land V$,   is the quotient of the tensor algebra on $V$ by the ideal generated by the relations $x\otimes y - (-1)^{\vert x\vert\cdot \vert y\vert} y\otimes x$.

A cdga $(A,d)$ is a \emph{minimal Sullivan algebra} if $A = \land V$, where $V=V^{\geq 1}$ is equipped with a filtration $V = \cup_n V(n)$ such that $d \colon  V\to \land^{\geq 2}V$ and $dV(n)\subset \land V(n-1)$. 

A basic theorem in rational homotopy (\cite{FHTII}) states that for each $(A,d)\in \catcdga$, there is a minimal Sullivan algebra $(\land V,d)$ equipped with a quasi-isomorphism   $\varphi \colon  (\land V,d)\to (A,d)$. Then, $(\land V,d)$ is called the \emph{minimal Sullivan model} of $(A,d)$.   
 
 Denote by $\CDGC$ the category of cocommutative differential graded coalgebras (cdgc's hereafter) 
 with degree $-1$ differential and which are connected, i.e., of the form $C=\Q\oplus C_{\geq1}$. 
 Denote by $\catdgl$ the category of differential graded Lie algebras which are concentrated in nonnegative degrees, 
 with differential of degree -1. 
 Following Quillen (\cite{Q}),   Neisendorfer (\cite{Neis}) defines adjoint functors 
 $$\xymatrix{
 \CDGC \ar@/^/[rr]^{\mathcal L} &&  \catdgl \ar@/^/[ll]^{\cC}}$$
  and proves that the adjunction maps $C \to {\mathcal C}{\mathcal L}(C)$ and ${\mathcal L}{\mathcal C}(L)\to L$ are quasi-isomorphisms.
Here, $\cC(L,d)$ is the usual coalgebra chain complex of $(L,d)$. If $(C, \delta)$ is a cdgc with diagonal $\Delta$, then 
one sets
$${\mathcal L}(C, \delta) = (\mathbb L(s^{-1}\overline{C}), d_1+d_2)\,,$$ 
where  
$d_1\colon s^{-1}\overline{C}\to s^{-1}\overline{C}$ is the desuspension of $\delta$ and $d_2$ is determined by
 $$d_2(s^{-1}a) = -\sum_i (-1)^{\vert a_i\vert} [s^{-1}a_i, s^{-1}a_i'],$$ 
 with
  $\Delta (a) = a\otimes 1+ 1\otimes a + \sum_i (a_i\otimes a_i' + (-1)^{\vert a_i\vert \cdot \vert a_i'\vert} a_i'\otimes a_i)$.
In fact, ${\mathcal L}(C,\delta)$ consists of the primitive elements in the cobar construction on $(C,\delta)$.

Now denote by $n{\catcdga}$ the full subcategory of $\catcdga$ formed by those cdga's admitting a minimal model $(\land Z,d)$ in which $Z$ is a finite type graded vector space. Denote also by n$\catdgl$ the full subcategory of $\catdgl$ formed by those dgl's $(L,d)$ with  finite type pronilpotent homology.
 Then, the associated homotopy categories are equivalent (\cite[Proposition 7.3]{Neis}) 
 (By definition the homotopy category h$_0$n$\catdgl$ is  the subcategory of h$_0$$\catdgl$ generated by the objects in n$\catdgl$, and similarly for h$_0$n$\catcdga$). 
 
 Each dgl $(\mathbb L(V),d)$ in n$\catdgl$ is quasi-isomorphic to a dgl of the form ${\mathcal L}(\land Z,d)^\#$, where $(\land Z,d)$ is a minimal Sullivan algebra with $Z= Z^{\geq 1}$ of finite type, and $(\land Z,d)^\#$ is the dual differential coalgebra,
 $$(\mathbb L(V),d) \simeq  {\mathcal L}(\land Z,d)^\#\,.$$
 
 A complete dgl  $(\widehat{\mathbb L}(V),d)$ is said to be \emph{minimal} if $dV\subset  \hL^{\geq 2}(V)$. 
Clearly, a morphism  of complete dgl's, $f\colon (\widehat{\mathbb L}(V),d)\to (\widehat{\mathbb L}(W),d)$,
is a quasi-isomorphism if, and only if, the induced map on the homology of indecomposable elements is an isomorphism. 
It follows that each complete dgl   $(\widehat{\mathbb L}(V),d)$ admits a unique minimal model.

In our situation, since $(\mathbb L(V),d)\simeq {\mathcal L}(\land Z,d)^\#$,  
the homologies of the indecomposable subspaces are isomorphic, and thus their completions are also quasi-isomorphic,
$$(\widehat{\mathbb L}(V),d)\simeq \widehat{\mathcal L}(\land Z,d)^\#.$$
Therefore, the injection $(\mathbb L(V),d)\to (\widehat{\mathbb L}(V),d)$ is a quasi-isomorphism if, and only if,
the corresponding  injection for ${\mathcal L}(\land Z,d)^\#$ is a quasi-isomorphism. 

\begin{proposition}\label{prop:LLhat} Let   $(L,d) = {\cL}(\land Z,d)^\#$, where
$(\land Z,d)$ is a finite type minimal Sullivan algebra with $Z = Z^{\geq 1}$. Then the injection 
$$(L,d) \to (\widehat{L},d)$$
is a quasi-isomorphism.\end{proposition}

\begin{proof} 
The elements of $s^{-1}Z^{\#}$ being cycles in $(L,d)$ and in $(\widehat{L},d)$, we get 
the following commutative diagram of complexes, in which the horizontal arrow is a morphism of dgl's,
$$\xymatrix{
{\cL}((\land Z,d)^\#)=(L,d)\ar[rr]
&&
(\widehat{L},d)\\
&
(s^{-1}Z^{\#},0).
\ar@{_{(}->}[lu] \ar@{^(->}[ru]
}$$
The proposition  is now a direct consequence of the next three lemmas.\end{proof}

\begin{lemma}\label{Lema2}
 The injection $( s^{-1}Z^\#,0)  \to {\mathcal L}(\land Z,d)^\#$ is a quasi-isomorphism.
 \end{lemma}
 
 \begin{proof}
Begin with the particular case $(\land Z,0)$.
We associate to $\land Z$ the algebra $\land Z'$ where the degrees of $Z$ have been multiplied by $3$,
i.e., $Z^n \cong Z'^{3n}$. 
Therefore $\land Z$ is isomorphic to $\land Z'$ as an ungraded algebra and $Z' = (Z')^{\geq 3}$. 
Since $(\land Z',0)$ is finite type and simply connected,  the injection
$$(s^{-1}(Z')^\#,0)\to {\mathcal L}(\land Z',0)^\#$$ is a quasi-isomorphism (\cite[Proposition 4.2(a)]{Neis}, 
\cite[Proposition 1.3.9]{Tan}). 
As $(\land Z,0)$ and $(\land Z',0)$ are isomorphic differential algebras, 
${\mathcal L}(\land Z,0)^\#$ and ${\mathcal L}(\land Z',0)^\#$ are isomorphic as differential Lie algebras. 
Therefore, the injection
\begin{equation}\label{1}
 (s^{-1}Z^\#,0)\to {\mathcal L}(\land Z,0)^\#
 \end{equation}
 is also a quasi-isomorphism.

\medskip
Now in the general case, $({\mathbb L}(W),D):={\mathcal L}(\land Z,d)^\#$, the differential $D$ is the sum 
$d_1+d_2$ where $d_2$ is the differential of ${\mathcal L}(\land Z,0)^\#$ and $d_1$ is the desuspension of the dual of $d$. 
 From (\ref{1}), we know that each $d_2$-cycle in ${\mathbb L}^r(W)$  is a $d_2$-boundary if $r>1$ 
 and is an element of $s^{-1}Z^\#$ if $r= 1$.  
 Let $x = \sum_{i= 1}^rx_i$ be a cycle in $(\mathbb L(W), D)$, with $x_i \in \mathbb L^i(W)$. Since $d_2x_r= 0$, 
 there is an element $y$ such that $x_r = d_2y$. 
 Replacing $x$ by $x- Dy$ we can suppose recursively that $r= 1$. 
 The result follows. 
\end{proof}

The Quillen adjunction $({\mathcal L}, \cC)$ has been dualized 
by Sinha and Walter (\cite{SW})  in an adjunction between $\catcdga$  and $\catdglco$
 (the category of differential graded Lie coalgebras, dglc's for short).
Recall that a graded Lie coalgebra is a graded vector space $E$ together with
 a comultiplication $\Delta\colon E\to E\otimes E$ that satisfies two properties,
$$(1+\tau)\circ \Delta = 0\quad \mbox{and} \quad
(1+\sigma + \sigma^2)\circ (1\otimes \Delta)\circ \Delta= 0.$$ 
Here, $\tau {\colon} E\otimes E\to E\otimes E$ is the graded permutation and 
$\sigma{\colon} E\otimes E\otimes E\to E\otimes E\otimes E$ is the graded cyclic permutation (\cite{Mi}).

Any graded coalgebra $(C, \Delta)$ admits a graded Lie coalgebra structure defined by $\Delta_L = \Delta-\tau \circ\Delta$.
If $T^c(V)$ denotes the tensor coalgebra on a graded vector space $V$, then $\Delta_L$ induces a Lie coalgebra structure
on the indecomposables for the shuffle product, i.e., on the quotient of $T^c(V)$ by the set 
of decomposable elements for the shuffle products.
This quotient is called the {\em free Lie coalgebra on $V$}  and  denoted by $\Lc(V)$.
The adjunction can be written as 
$$
\xymatrix{
\catcdga\ar@/_{}/[rr]_{\mathcal E}
&&
 \catdglco.   \ar@/_/[ll]_{\mathcal A} }$$
If $(E,d)$ is a differential graded Lie coalgebra, then, ${\mathcal A}(E,d)= (\land sE,D)$ with
$$D(sx) = \frac{1}{2} \sum_i (-1)^{\vert x_i\vert} \,\, sx_i\land sx_i' - sdx,$$
if $\Delta x= \sum_i x_i\otimes x_i'$. 
On the other hand, for an augmented cdga $(A,d)$, the dglc ${\mathcal E}(A,d)$ is the quotient of the bar construction 
on $(A,d)$ by the set 
of decomposable elements for the shuffle product.

 \begin{lemma}
There is an isomorphism of differential graded Lie algebras
 $$(\widehat{L},d) \cong \left( {\mathcal E}(\land Z,d)\right)^\#\,.$$
\end{lemma}

\begin{proof}
We have
$$\left[{\mathcal E}(\land Z,d)\right]^n = 
\bigoplus_{\left\{(q,r)\vert q(r-1) = n\right\}} (\mathbb L^c)^q(s^{-1}(\land Z)^r).$$
Therefore,
\begin{eqnarray*}
 \left[ {\mathcal E}(\land Z,d)\right]^\#_n
 &=&
 \prod_{\left\{(q,r)\vert q(r-1) = n\right\}} \left[ (\mathbb L^c)^q(s^{-1}(\land Z)^r)\right]^\#
 \\
 &=&
 \prod_{\left\{(q,r)\vert q(r-1) = n\right\}} \mathbb L^q(s^{-1}((\land Z)^\#)_r) = \widehat{{\mathbb L}}((\land Z)^\#)_n\,.
\end{eqnarray*}
\end{proof}

\begin{lemma}
 The injection $( s^{-1}Z^\#,0) \to (\widehat{L},d) $ is a quasi-isomorphism.
 \end{lemma}
 
 \begin{proof}
 It is equivalent to prove that the dual projection
 $$q\colon {\mathcal E}(\land Z,d) \to (s^{-1}Z,0)$$
 is a quasi-isomorphism.
 
 \medskip
 As in the proof of Lemma \ref{Lema2}, we begin with the particular case $(\land Z,0)$ and still consider the algebra
 $\land Z'$ with $Z^n\cong {Z'}^{3n}$, for which the injection 
 $(s^{-1}(Z')^\#,0) \hookrightarrow {\mathcal L}((\land Z',0)^\#)$
  is a quasi-isomorphism. Therefore,
 as $ \land Z$ and $\land Z'$ are isomorphic, 
 ${\mathcal E}(\land Z,0)$ and ${\mathcal E}(\land Z',0)$ are isomorphic and the  projection 
 $$q \colon {\mathcal E}(\land Z, 0) \to (s^{-1}Z,0)$$
is a quasi-isomorphism.
 
\medskip
In the general case, write $(\mathbb L^c(V),D)= {\mathcal E}(\land Z,d)$ with 
$D= d_1+d_2$ and observe that
$({\mathbb L}^c(V), d_2) = {\mathcal E}(\land Z,0)$. 
By definition of a Sullivan minimal model, there is a filtration $Z = \cup_n T(n)$, with $T(0)= \mbox{ker}\, d$,
and  $T(n) = Z\cap d^{-1}(\land T(n-1))$ for $n\geq 1$. 
 If $Z(n)$ is a supplement of $T(n-1)$ in $T(n)$, then we have
 $Z = \oplus_{n\geq 0} Z(n)$ and  each monomial of $\land Z$ can be written as
 $\omega= u_1\land \cdots \land u_r$ with $u_i\in    Z(n_i)$. 
 We define \emph{the multi-index of the monomial} $\omega$ as the sequence $(\dots ,0, q_n, q_{n-1},\dots , q_1, q_0)$,
  where $q_s$ denotes the number of elements $u_i\in Z(s)$.
  (Each sequence can be continued with a sequence of $0$ on the left.)
   We order these sequences with the lexicographic order:
$$(q_n, q_{n_1},\dots , q_0) > (\ell_n, \ell_{n-1},\dots , \ell_0)$$
if for some $r$, we have $q_p= \ell_p$ for $p>r$ and $q_r>\ell_r$.

In the bar construction on $(\land Z,d)$ and its quotient ${\mathcal E}(\land Z,d)$, we define the multi-index of an element 
$[\omega_1\vert \dots \vert \omega_r]$ as the sum of the multi-indices of the $\omega_i$. 
Finally the multi-index of a sum of monomials is the maximum of the multi-indices of the monomials.
By construction, the part $d_{2}$ of the differential $D$ preserves the multi-index and the differential
$d_{1}=s^{-1}d$ decreases the multi-index.

Let $x$ be a cycle in the kernel of the projection $q$. We write $x$ as a finite sum
$x = \sum_{i= 0}^N x_i$, with $x_i \in (\mathbb L^c)^{i+1}(V)$. 
From $d_2x_0= 0$ and the first part of the proof corresponding to $d=0$, we know the existence of
$y_{1}\in (\mathbb L^c)^2(V)$ such that  $x_0= d_2y_1$. 
Replacing $x $ by $x-Dy_1$ we can suppose that $x = \sum_{i=1}^N x_i$. 
Then $d_2x_1= 0$ and recursively we can suppose that $x= x_N$. 
The same argument gives $d_2x_N = 0$ and so $x_N= d_2y_{N+1}$.  
Therefore we may replace $x$ by $x'=d_1y_{N+1}$ and keep $[x]=[x']$. By the discussion above, 
we know that $y_{N+1}$ has the same multi-index than $x$ and that $d_1y_{N+1}$ has a lower multi-index. 
So by iterating this operation we get  a sequence of classes of higher length of words  and lower multi-indices. 
Therefore after some steps, we obtain an element $y_{q}$ equal to $0$, which shows that $x$ is a boundary. 
\end{proof}

\begin{proof}[Proof of \thmref{thm:main}]
Proposition \ref{Prop1} establishes one implication of the equivalence.
As for the second one, let $(\L(V),d)$ be a finite type pronilpotent graded Lie algebra, of completion
$(\widehat{\mathbb L}(V),d)$.
As it was already observed, we deduce from \cite{Neis} that $(\L(V),d)$ is quasi-isomorphic to
${\mathcal L}(\land Z,d)^\#$,
where $(\land Z,d)$ is a minimal Sullivan algebra. The result now follows from Proposition \ref{prop:LLhat}.
\end{proof}

\section{DGLie models for spaces}  \label{Sec3}  
 
 The association of a rational dgl to a space has a long history.
 In his pioneer work in 1969, Quillen associates to each simply-connected pointed topological space, $X$,  a 
dgl $\lambda(X)$,
whose homology is isomorphic to the rational homotopy Lie algebra of $X$, 
$H_*({\lambda}(X)) \cong \pi_*(\Omega X)\otimes \mathbb Q$ (\cite{Q}). 
This association defines a functor inducing  an equivalence of  categories between   
the homotopy category of rational 
simply connected pointed spaces, and   the homotopy category of  rational dgl's $(L,d)$ with $L= L_{\geq 1}$.

\medskip
Some time after,  in \cite{Su}, Sullivan establishes another equivalence of  categories, 
between the homotopy category of pointed rational nilpotent spaces with finite Betti numbers 
and the homotopy category of  rational augmented commutative differential graded algebras whose homologies are
connected and of finite type. The starting point of the construction is the simplicial cdga $A_{PL}^{\bullet}$
where 
$$A_{PL}^n = \land (t_0, \dots , t_n, dt_0, \dots dt_n)/ (\sum_{i=0}^n t_i-1, \sum_{i=0}^n dt_i)\,.$$
Here, the elements $t_{i}$ are of degree
0 and the $dt_{i}$'s of degree 1. 
For a simplicial set $ {K}$, the cdga $A_{PL}(K)$ is defined by  
$A_{PL}(K) = \mbox{Hom}_{\mbox{\scriptsize SSets}} ( {K}, A_{PL}^{\bullet})$. 
If $X$ is a topological space, we denote by $A_{PL}(X)$ the cdga associated to the simplicial set of singular simplices
of $X$. The minimal model $m_X$ of $X$ is the minimal model of $A_{PL}(X)$. 

When $X$ is simply connected, Majewski proves (\cite{Ma}) that 
the Quillen model ${\lambda}(X)$ is quasi-isomorphic to the dgl 
${\mathcal L}(m_X^\#)$.
 
When $X$ is nilpotent with finite Betti numbers,    $m_X$ is a finite type cdga, and Neisendorfer uses the dgl ${\mathcal L}(m_X^\#) $ as a Lie model for $X$.

\medskip
More recently, a new construction has been made from ideas of E. Getzler (\cite{Getz}) 
and R. Lawrence and D. Sullivan (\cite{LS}).   
 When $X$ is a connected finite simplicial complex, a transfer diagram has been built in the category of cdga's
$$
\xymatrix{
{\fracd}\colon
&
 \ar@(ul,dl)@<-5.5ex>[]_{\phi}
 &
  A_{PL}(X) \ar@<0.75ex>[r]^-{p}
  &
   {C^*(X),} \ar@<0.75ex>[l]^-{i} }
$$
where $C^*(X)$ denotes the simplicial cochain complex on $X$. This transfer diagram induces a differential $d$ on  ${\mathbb L}^c(s^{-1}C^*(X))$ turning ${\lasu}^c(X):= ({\mathbb L}^c(s^{-1}C^*(X)),d)$ into a differential graded Lie coalgebra equipped with a quasi-isomorphism
$$ I \colon  \overline{{\lasu}}^c(X)\to {\mathcal E}(A_{PL}(X))\,.$$
Here, $\overline{{\lasu}}^c(X) $ is the kernel of the map ${{\lasu}}^c(X)\to  {{\lasu}}^c(x_0)$ induced by the injection of a base point $x_0\in X$. (\cite{Getz}, \cite{BFMT}). 

\begin{definition}   The \emph{Sullivan Lie model} of $X$ is the dgl 
${\lasu}_X := ({\lasu}^c(X))^\# $, and its \emph{reduced Sullivan Lie model} is the    quotient 
$$\overline{{\lasu}}_X := {\lasu}_X/ {{\lasu}}_{(x_0)}.$$ 
\end{definition}

Since the functor ${\mathcal E}$ preserves quasi-isomorphisms, we get   quasi-isomorphisms
$$ \ov{\lasu}_X  \simeq ({\mathcal E}(m_X))^\#  \,.$$
From \thmref{thm:main}, we then deduce:

\begin{theoremb}\label{thm:neisendorfer}
 The Sullivan Lie model of  a connected finite simplicial complex $X$ is a generalization of the Neisendorfer model,
 i.e., if $X$ is nilpotent there are quasi-isomorphisms
 $$ \ov{\lasu}_X \simeq ({\mathcal E}(m_X))^\# \simeq \cL(m_{X}^{\#})
 .$$
 \end{theoremb}

\medskip
We thank an anonymous referee for pointing out an error in a first version of \thmref{thm:main} and for suggesting
the example of \remref{rem:simple}.

\providecommand{\bysame}{\leavevmode\hbox to3em{\hrulefill}\thinspace}
\providecommand{\MR}{\relax\ifhmode\unskip\space\fi MR }
\providecommand{\MRhref}[2]{%
  \href{http://www.ams.org/mathscinet-getitem?mr=#1}{#2}
}
\providecommand{\href}[2]{#2}

\end{document}